\documentclass{article}

\usepackage{amssymb, amsmath, amsthm, tikz, verbatim, graphicx,lmodern,indentfirst,enumerate,color}

\usepackage[titletoc,page]{appendix}
\usepackage[margin=1.4in]{geometry}

\newtheorem{thm}{Theorem}[section]
\newtheorem{lem}[thm]{Lemma}
\newtheorem{prop}[thm]{Proposition}

\newtheorem{oprob}[thm]{Open Problem}

\theoremstyle{definition}
\newtheorem{example}[thm]{Example}

\theoremstyle{definition}
\newtheorem{question}[thm]{Question}

\theoremstyle{definition}
\newtheorem{defn}[thm]{Definition}

\theoremstyle{remark}
\newtheorem{rem}[thm]{Remark}

\numberwithin{equation}{section}

\newcommand*\circled[1]{\tikz[baseline=(char.base)]{
            \node[shape=circle,draw,inner sep=0pt,minimum size=5mm] (char) {#1};}}

\newcommand{\rootsD}[5]{\circled{#1}
            \begin{tabular}{ccc}
             &#2& \\
             #3&#4&#5
             \end{tabular}}
\makeatletter
\newcommand{\rmnum}[1]{\romannumeral #1}
\newcommand{\Rmnum}[1]{\expandafter\@slowromancap\romannumeral #1@}
\makeatother
\begin{document}

\title{Complete reducibility of subgroups of reductive algebraic groups over nonperfect fields \Rmnum{3}}
\author{Tomohiro Uchiyama\\
Faculty of International Liberal Arts\\
1-236, Tangi, Hachioji, Tokyo, Soka University, Japan\\
\texttt{email:t.uchiyama2170@gmail.com}}
\date{}
\maketitle 

\begin{abstract}
Let $G$ be a reductive group over a nonperfect field $k$. We study rationality problems for Serre's notion of complete reducibility of subgroups of $G$. In our previous work, we constructed examples of subgroups $H$ of $G$ that are $G$-completely reducible but not $G$-completely reducible over $k$ (and vice versa). In this paper, we give a theoretical underpinning of those constructions. Then using Geometric Invariant Theory, we obtain a new result on the structure of $G(k)$-(and $G$-) orbits in an arbitrary affine $G$-variety. We discuss several related problems to complement the main results. 
\end{abstract}

\section{Introduction}
Let $k$ be a field. We write $\overline k$ for an algebraic closure of $k$. Let $G$ be a (possibly non-connected) affine algebraic $k$-group: we regard $G$ as a $\overline k$-defined algebraic group together with a choice of $k$-structure in the sense of Borel~\cite[AG.~11]{Borel-AG-book}. We say that $G$ is \emph{reductive} if the unipotent radical $R_u(G)$ of $G$ is trivial. Throughout, $G$ is always a (possibly non-connected) reductive $k$-group. In this paper, we continue our study of rationality problems for complete reducibility of subgroups of $G$~\cite{Uchiyama-Nonperfect-pre},~\cite{Uchiyama-Nonperfectopenproblem-pre}. By a subgroup of $G$ we mean a (possibly non-$k$-defined) closed subgroup of $G$. If a subgroup of $G$ needs to be $k$-defined (or $G$ needs to be connected) in some statement, we explicitly say so. Recall~\cite[Sec.~3]{Serre-building}.
\begin{defn}
A subgroup $H$ of $G$ is called \emph{$G$-completely reducible over $k$} (\emph{$G$-cr over $k$} for short) if whenever $H$ is contained in a $k$-defined $R$-parabolic subgroup $P$ of $G$, then $H$ is contained in a $k$-defined $R$-Levi subgroup of $P$. In particular if $H$ is not contained in any proper $k$-defined $R$-parabolic subgroup of $G$, $H$ is called \emph{$G$-irreducible over $k$} (\emph{$G$-ir over $k$} for short). We simply say that $H$ is \emph{$G$-cr} (\emph{$G$-ir}) if $H$ is $G$-cr over $\overline k$ ($G$-ir over $\overline k$). 
\end{defn}
We define $R$-parabolic subgroups and $R$-Levi subgroups in the next section (Definition~\ref{R-parabolic}). These concepts are essential to extend the notion of complete reducibility (initially defined only for subgroups of connected $G$~\cite[Sec.~3]{Serre-building}) to subgroups of non-connected $G$~\cite{Bate-nonconnected-PAMS},~\cite[Sec.~6]{Bate-geometric-Inventione}. We defined complete reducibility for a possibly non-$k$-defined subgroup of $G$. This is because for a subgroup $H$ of $G$, some closely related important subgroups of $G$ are not necessarily $k$-defined even if $H$ is $k$-defined. For example, centralizers or normalizers of $k$-subgroups of $G$ are not necessarily $k$-defined; see~\cite[Thm.~1.2 and Thm.~1.7]{Uchiyama-Nonperfect-pre} for such examples. If $G$ is connected and $H$ is a subgroup of $G(k)$, our notion of complete reducibility agrees with the usual one of Serre.

So far, most work on complete reducibility has considered only the case $k=\overline k$; see~\cite{Liebeck-Seitz-memoir},~\cite{Stewart-nonGcr},~\cite{Thomas-irreducible-JOA} for example. Not much is known on complete reducibility over $k$ for general fields $k$ (especially for nonperfect $k$) except for a few theoretical results and important examples in~\cite[Sec.~5]{Bate-geometric-Inventione},~\cite{Bate-cocharacter-Arx},~\cite{Uchiyama-Nonperfect-pre},~\cite{Uchiyama-Nonperfectopenproblem-pre}. In particular, in~\cite[Thm.~1.10]{Uchiyama-Separability-JAlgebra}~\cite[Thm.~1.8]{Uchiyama-Classification-pre},~\cite[Thm.~1.2]{Uchiyama-Nonperfect-pre}, we found several examples of $k$-subgroups of $G$ that are $G$-cr over $k$ but not $G$-cr (and vice versa). The main result of this paper is to give a theoretical underpinning for our (possibly somewhat mysterious) construction of those examples. For an algebraic extension $k'$ of $k$ and an affine group $N$, we denote the set of $k'$-points of $N$ by $N(k')$. For a subgroup $H$ of $G$, we write $g\cdot H$ for $g H g^{-1}$ where $g\in G$.  

\begin{thm}\label{Prop1}
Let $k$ be a nonperfect separably closed field. Suppose that a subgroup $H$ of $G$ is $G$-cr but not $G$-cr over $k$. Let $P$ be minimal among $k$-defined $R$-parabolic subgroups of $G$ containing $H$. Then there exists a unipotent element $u\in R_u(P)(\overline k)$ such that $u\cdot H$ is $G$-cr over $k$. 
\end{thm}
  
\begin{thm}\label{Prop2}
Let $k$ be a nonperfect separably closed field. Suppose that a subgroup $H$ of $G$ is $G$-cr over $k$ but not $G$-cr. Let $P$ be minimal among $k$-defined $R$-parabolic subgroups of $G$ containing $H$. Then 
\begin{enumerate}
\item $H$ is $L$-ir over $k$ for some $k$-defined $R$-Levi subgroup $L$ of $P$. 
\item Moreover, there exists an element $l\in L(\overline k)$ such that $l\cdot H$ is not $G$-cr over $k$.  
\end{enumerate}
\end{thm}
To illustrate our theoretical results (Theorems~\ref{Prop1},~\ref{Prop2}) and ideas in the proofs, we present a new example of a $k$-subgroup of $G$ that is $G$-cr over $k$ but not $G$-cr (and vice versa).
\begin{thm}\label{D4example}
Let $k$ be a nonperfect separably closed field of characteristic $2$. Let $\tilde G$ be a simple $k$-group of type $D_4$. Let $\sigma$ be a non-trivial element of order $3$ in the outer automorphism group of $\tilde G$. 
Let $G:=\tilde G\rtimes \langle \sigma \rangle$. Then there exists a $k$-subgroup $H$ of $G$ that is $G$-cr over $k$ but not $G$-cr (and vice versa).
\end{thm}
A few comments are in order. First, the non-perfectness of $k$ is (almost) essential in Theorems~\ref{Prop1},~\ref{Prop2}, and~\ref{D4example} in view of the following~\cite[Thm.~1.1]{Bate-separable-Paris}:
\begin{prop}\label{paris}
Let $G$ be connected. Let $H$ be a subgroup of $G$. Then $H$ is $G$-cr over $k$ if and only if $H$ is $G$-cr over $k_s$. 
\end{prop}
So in particular if $k$ is perfect and $G$ is connected, a subgroup of $G$ is $G$-cr over $k$ if and only if it is $G$-cr. The forward direction of Proposition~\ref{paris} holds for non-connected $G$. The reverse direction depends on the recently proved \emph{center conjecture of Tits}~\cite{Serre-building},~\cite{Tits-colloq},~\cite{Weiss-center-Fourier} in spherical buildings, but this method does not work for non-connected $G$; the set of $R$-parabolic subgroups does not form a simplicial complex in the usual sense of Tits~\cite{Tits-book} as we have shown in~\cite[Thm.~1.12]{Uchiyama-Nonperfectopenproblem-pre}. In the following we assume that $k$ is separably closed. So every maximal $k$-torus of $G$ splits over $k$, thus $G$ is $k$-split. This simplifies arguments in many places. For the theory of complete reducibility over arbitrary $k$, see~\cite{Bate-cocharacter-Arx},~\cite{Bate-cocharacterbuildings-Arx}.

Second, note that the $k$-definedness of $H$ in Theorem~\ref{D4example} is important. 
Actually it is not difficult to find a $\overline k$-subgroup with the desired property. For our construction to work, it is essential for $H$ to be $\emph{nonseparable}$ in $G$. We write $\textup{Lie}(G)$ or $\mathfrak g$ for the Lie algebra of $G$. Recall~\cite[Def.~1.1]{Bate-separability-TransAMS}.
\begin{defn}
A subgroup $H$ of $G$ is \emph{nonseparable} if the dimension of $\textup{Lie}(C_G(H))$ is strictly smaller than the dimension of $\mathfrak{c}_{\mathfrak{g}}(H)$ (where $H$ acts on $\mathfrak g$ via the adjoint action). In other words, the scheme-theoretic centralizer of $H$ in $G$ (in the sense of~\cite[Def.~A.1.9]{Conrad-pred-book}) is not smooth. 
\end{defn} 
We exhibit the importance of nonseparability of $H$ in the proof of Theorem~\ref{D4example}. Proper nonseparable $k$-subgroups of $G$ are uncommon, and only a handful examples are known~\cite[Sec.~7]{Bate-separability-TransAMS},~\cite[Thm.~1.10]{Uchiyama-Separability-JAlgebra}~\cite[Thm.~1.8]{Uchiyama-Classification-pre},~\cite[Thm.~1.2]{Uchiyama-Nonperfect-pre}. Note that if characteristic of $k$ is very good for connected $G$, every subgroup of $G$ is separable~\cite[Thm.~1.2]{Bate-separability-TransAMS}. Thus, to find a nonseparable subgroup we are forced to work in small $p$ (at least for connected $G$). See~\cite{Bate-separability-TransAMS},~\cite{Herpel-smoothcentralizerl-trans} for more on separability. 

Our second main result is a generalization of Theorems~\ref{Prop1} and~\ref{Prop2} using the language of Geometric Invariant Theory (GIT for short)~\cite{Mumford-GIT-book}. Let $V$ be a (possibly non-connected) affine $k$-variety. When $G$ acts on $V$ $k$-morphically, we say that $V$ is a $G$-variety. One of the main themes of GIT is to study the structure of $G$-orbits (and $G(k)$-orbits) in $V$~\cite{Kempf-instability-Ann},~\cite{Bate-uniform-TransAMS},~\cite{Bate-cocharacter-Arx}. Recently studies on complete reducibility (over $k$) via GIT have been very fruitful; GIT gives a very short and uniform proof for many results~\cite{Bate-geometric-Inventione},~\cite{Bate-uniform-TransAMS},~\cite{Bate-cocharacter-Arx}. This makes a striking contrast to traditional representation theoretic methods (which depend on long case-by-case analyses)~\cite{Liebeck-Seitz-memoir},~\cite{Stewart-nonGcr},~\cite{Thomas-irreducible-JOA}. 

We recall the following algebro-geometric characterization for complete reducibility (over $k$) via GIT (\cite[Prop.~2.16, Thm.~3.1]{Bate-geometric-Inventione} and~\cite[Thm.~9.3]{Bate-cocharacter-Arx}). This turns problems on complete reducibility into problems on the structure of $G(k)$-(or $G$-) orbits. Let $H$ be a subgroup of $G$ such that $H=\langle h_1,\cdots, h_n \rangle$ for some $n\in \mathbb{N}$ and $\textbf{h}:=(h_1,\cdots,h_n)\in G^n$. Suppose that $G$ (and $G(k)$) acts on $G^n$ via simultaneous conjugation. 
\begin{prop}\label{geometric}
$H$ is $G$-cr if and only if $G\cdot \textbf{h}$ is Zariski closed in $G^n$. Moreover, $H$ is $G$-cr over $k$ if and only if $G(k)\cdot \textbf{h}$ is cocharacter closed over $k$.  
\end{prop}
The definition of a cocharacter closed orbit is given in the next section (Definition~\ref{cocharacterclosure}). Using Proposition~\ref{geometric} and various techniques from GIT we can sometimes generalize results on complete reducibility (over $k$) to obtain new results on GIT where $G$ (or $G(k)$) acts on an arbitrary affine $G$-variety rather than on some tuple of $G$; see~\cite{Bate-uniform-TransAMS},~\cite{Bate-cocharacter-Arx} for example. We follow the same philosophy here and generalize Theorems~\ref{Prop1} and~\ref{Prop2}.

\begin{thm}\label{GIT1}
Let $k$ be nonperfect. Suppose that there exists $v\in V$ such that $G\cdot v$ is Zariski closed but $G(k)\cdot v$ is not cocharacter closed over $k$. Let $\Delta_{v,k}$ be the set of $k$-cocharacters of $G$ destabilizing $v$ over $k$. Pick $\lambda\in \Delta_{v,k}$ such that $P_\lambda$ is minimal among $R$-parabolic subgroups $P_\mu$ for $\mu\in \Delta_{v,k}$. Then there exists a unipotent element $u\in R_u(P_\lambda)(\overline k)$ such that $G(k)\cdot (u\cdot v)$ is cocharacter closed over $k$.
\end{thm}

\begin{thm}\label{GIT2}
Let $k$ be nonperfect. Suppose that there exists $v\in V$ such that $G(k)\cdot v$ is cocharacter closed over $k$ but $G\cdot v$ is not Zariski closed. Let $\Delta_{v,k}$ be the set of $k$-cocharacters of $G$ destabilizing $v$ over $k$. Pick $\lambda\in \Delta_{v,k}$ such that $P_\lambda$ is minimal among $R$-parabolic subgroups $P_\mu$ for $\mu\in \Delta_{v,k}$. Then
\begin{enumerate}
\item{There exists $\xi \in \Delta_{v,k}$ such that $P_\lambda=P_\xi$ and $\xi(\overline k^*)$ fixes $v$.}
\item{Any $k$-defined cocharacter of $L_{\xi}$ destabilizing $v$ over $k$ is central in $L_{\xi}$.}
\item{Moreover, there exists an element $\l\in L_{\xi}(\overline k)$ such that $G(k)\cdot(l\cdot v)$ is not cocharacter closed over $k$.}
\end{enumerate}
\end{thm}
Roughly speaking, we say that $v\in V$ is destabilized over $k$ by a $k$-cocharacter $\lambda$ of $G$ if $v\in V$ is taken outside of $G(k)\cdot v$ by taking a limit of $v$ along $\lambda$ in the sense of GIT~\cite{Kempf-instability-Ann},~\cite{Mumford-GIT-book}; see Definition~\ref{destabilizingcocharacter} for the precise definition. Note that if $k$ is perfect, Theorems~\ref{GIT1} and~\ref{GIT2} have no content: in that case $G(k)\cdot v$ is cocharacter closed if and only if $G\cdot v$ is Zariski closed~\cite[Cor.~7.2, Prop.~7.4]{Bate-cocharacter-Arx}.

To complement the paper we also investigate the structure of centralizers of completely reducible subgroups of $G$. In particular we ask~\cite[Open Problem~1.4]{Uchiyama-Nonperfectopenproblem-pre}:
\begin{oprob}\label{centprob}
Suppose that a $k$-subgroup $H$ of $G$ is $G$-cr over $k$. Is $C_G(H)$ also $G$-cr over $k$?
\end{oprob}
We have a partial answer~\cite[Thm.~1.5]{Uchiyama-Nonperfectopenproblem-pre}:
\begin{prop}\label{cent}
Let $G$ be connected. Suppose that a $k$-subgroup $H$ of $G$ is $G$-cr over $k$. 
If $C_G(H)$ is reductive, then it is $G$-cr over $k$.
\end{prop}
We need the connectedness assumption in Proposition~\ref{cent} since it depends on the center conjecture of Tits. If $k=\overline k$ (or more generally if $k$ is perfect), the answer to Open Problem~\ref{centprob} is ``yes'' by~\cite[Cor.~3.17]{Bate-geometric-Inventione} (and Proposition~\ref{paris}). Some trouble arises for nonperfect $k$ since $C_G(H)$ is not necessarily reductive even if $H$ is $G$-cr over $k$~\cite[Rem.~3.11]{Uchiyama-Nonperfect-pre}. This does not happen if $k=\overline k$ by~\cite[Prop.~3.12]{Bate-geometric-Inventione}, which in turn depends on a deep result of Richardson~\cite[Thm.~A]{Richardson-coset-BLM}. The reductivity of $C_G(H)$ was crucial in the proof of~\cite[Cor.~3.17]{Bate-geometric-Inventione} to apply a tool from GIT. 

In general, if $k$ is nonperfect, even if a $k$-subgroup of $G$ is $G$-cr over $k$, it is not necessarily reductive ~\cite[Prop.~1.10]{Uchiyama-Nonperfect-pre}. This pathology happens because the classical construction of Borel-Tits~\cite[Prop.~3.1]{Borel-Tits-unipotent-invent} fails over nonperfect $k$; see~\cite[Sec.~3.2]{Uchiyama-Nonperfect-pre}. This does not happen if $k=\overline k$; a $G$-cr subgroup is always reductive~\cite[Prop.~4.1]{Serre-building}. 

Here is our third main result in this paper. Let $G$ be connected. Fix a maximal $k$-torus $T$ of $G$. We write $\overline{w_{0,G}}$ for an automorphism of $G$ which normalizes $T$ and induces $-1$ on $\Psi^{+}(G)$ (the set of positive roots of $G$). It is known that $\overline{w_{0,G}}=w_{0,G}$ for $G$ of not type $A_n$, $D_{2n+1}$, or $E_6$, and $\overline{w_{0,G}}=w_{0,G} \sigma_G$ for $G$ of type $A_n (n>1)$, $D_{2n+1}$, or $E_6$ where $w_{0,G}$ is the longest element of the Weyl group of $G$ and $\sigma_G$ is a suitable graph automorphism of $G$ (cf.~\cite[Proof of Thm.~4.1]{Liebeck-Seitz-memoir}).

\begin{thm}\label{main}
Let $G$ be connected. Suppose that a semisimple $k$-subgroup $H$ of $G$ is $G$-cr over $k$. Let $P$ be minimal among $k$-parabolic subgroups containing $HC_G(H)$, and $L$ be a $k$-Levi subgroup of $P$. If the automorphism $\overline{w_{0,L}}$ of $L$ extends to an automorphism of $G$ and $\overline{w_{0,L}}$ stabilizes $R_u(P)$ (in particular if $L$ is not of type $A_n$, $D_{2n+1}$, or $E_6$), then $C_G(H)$ is $G$-cr over $k$. 
\end{thm}
Note that we give a discussion of the conditions of this theorem (automorphisms extending) in Remark~\ref{remOnAssumptions}.

Here is the structure of the paper. In Section 2, we set out the notation and show some preliminary results. Then in Section 3, we prove our first main result (Theorems~\ref{Prop1} and~\ref{Prop2}). In Section 4, we present the $D_4$ example~(Theorem~\ref{D4example}). In Section 5, we generalize Theorems~\ref{Prop1} and~\ref{Prop2} and prove our second main result (Theorems~\ref{GIT1} and~\ref{GIT2}). In Section 6, we translate Theorems~\ref{Prop1} and~\ref{Prop2} into the language of spherical buildings, and prove Propositions~\ref{building1} and~\ref{building2}. This gives a topological perspective for the rationality problems for complete reducibility and GIT. 
Then in Section 7, we attack Open Problem~\ref{centprob}, and prove Theorem~\ref{main} in a purely combinatorial way. In Section 8, we consider a problem on the number of conjugacy classes and prove Theorem~\ref{conjugacy}. We note that nonseparability comes into play in a crucial way in the proof of Theorem~\ref{conjugacy}.

\section{Preliminaries}
Throughout, we denote by $k$ a separably closed field. Our references for algebraic groups are~\cite{Borel-AG-book},~\cite{Borel-Tits-Groupes-reductifs},~\cite{Conrad-pred-book},~\cite{Humphreys-book1}, and~\cite{Springer-book}. 

Let $H$ be a (possibly non-connected) affine algebraic group. We write $H^{\circ}$ for the identity component of $H$. It is clear that if $H$ is $k$-defined, $H^{\circ}$ is $k$-defined. We write $[H,H]$ for the derived group of $H$. A reductive group $G$ is called \emph{simple} as an algebraic group if $G$ is connected and all proper normal subgroups of $G$ are finite. We write $X_k(G)$ and $Y_k(G)$ ($X(G)$ and $Y(G)$) for the set of $k$-characters and $k$-cocharacters ($\overline k$-characters and $\overline k$-cocharacters) of $G$ respectively. For $\overline k$-characters and $\overline k$-cocharacters $G$ we simply say characters and cocharacters of $G$. 

Fix a maximal $k$-torus $T$ of $G$ (such a $T$ exists by~\cite[Cor.~18.8]{Borel-AG-book}). Then $T$ splits over $k$ since $k$ is separably closed. Let $\Psi(G,T)$ denote the set of roots of $G$ with respect to $T$. We sometimes write $\Psi(G)$ for $\Psi(G,T)$. Let $\zeta\in\Psi(G)$. We write $U_\zeta$ for the corresponding root subgroup of $G$. We define $G_\zeta := \langle U_\zeta, U_{-\zeta} \rangle$. Let $\zeta, \xi \in \Psi(G)$. Let $\xi^{\vee}$ be the coroot corresponding to $\xi$. Then $\zeta\circ\xi^{\vee}:\overline k^{*}\rightarrow \overline k^{*}$ is a $k$-homomorphism hence $(\zeta\circ\xi^{\vee})(a) = a^n$ for some $n\in\mathbb{Z}$.
Let $s_\xi$ denote the reflection corresponding to $\xi$ in the Weyl group of $G$. Each $s_\xi$ acts on the set of roots $\Psi(G)$ by the following formula~\cite[Lem.~7.1.8]{Springer-book}:
$
s_\xi\cdot\zeta = \zeta - \langle \zeta, \xi^{\vee} \rangle \xi. 
$
\noindent By \cite[Prop.~6.4.2, Lem.~7.2.1]{Carter-simple-book} we can choose $k$-homomorphisms $\epsilon_\zeta : \overline k \rightarrow U_\zeta$  so that 
$
n_\xi \epsilon_\zeta(a) n_\xi^{-1}= \epsilon_{s_\xi\cdot\zeta}(\pm a)
            \text{ where } n_\xi = \epsilon_\xi(1)\epsilon_{-\xi}(-1)\epsilon_{\xi}(1).  \label{n-action on group}
$

We recall the notions of $R$-parabolic subgroups and $R$-Levi subgroups from~\cite[Sec.~2.1--2.3]{Richardson-conjugacy-Duke}. These notions are essential to define $G$-complete reducibility for subgroups of non-connected $G$ and also to translate results on complete reducibility into the language of GIT; see~\cite{Bate-nonconnected-PAMS} and~\cite[Sec.~6]{Bate-geometric-Inventione}. 

\begin{defn}
Let $X$ be a affine $k$-variety. Let $\phi : \overline k^*\rightarrow X$ be a $k$-morphism of affine $k$-varieties. We say that $\displaystyle\lim_{a\rightarrow 0}\phi(a)$ exists if there exists a $k$-morphism $\hat\phi:\overline k\rightarrow X$ (necessarily unique) whose restriction to $\overline k^{*}$ is $\phi$. If this limit exists, we set $\displaystyle\lim_{a\rightarrow 0}\phi(a) = \hat\phi(0)$.
\end{defn}

\begin{defn}\label{R-parabolic}
Let $\lambda\in Y(G)$. Define
$
P_\lambda := \{ g\in G \mid \displaystyle\lim_{a\rightarrow 0}\lambda(a)g\lambda(a)^{-1} \text{ exists}\}, $\\
$L_\lambda := \{ g\in G \mid \displaystyle\lim_{a\rightarrow 0}\lambda(a)g\lambda(a)^{-1} = g\}, \,
R_u(P_\lambda) := \{ g\in G \mid  \displaystyle\lim_{a\rightarrow0}\lambda(a)g\lambda(a)^{-1} = 1\}. 
$
\end{defn}
We call $P_\lambda$ an $R$-parabolic subgroup of $G$, and $L_\lambda$ an $R$-Levi subgroup of $P_\lambda$. Note that $R_u(P_\lambda)$ is the unipotent radical of $P_\lambda$. If $\lambda$ is $k$-defined, $P_\lambda$, $L_\lambda$, and $R_u(P_\lambda)$ are $k$-defined~\cite[Sec.~2.1-2.3]{Richardson-conjugacy-Duke}. All $k$-defined parabolic subgroups and $k$-defined Levi subgroups of $G$ arise in this way since $k$ is separably closed. It is well known that $L_\lambda = C_G(\lambda(\overline k^*))$. Note that $k$-defined $R$-Levi subgroups of a $k$-defined $R$-parabolic subgroup $P$ of $G$ are $R_u(P)(k)$-conjugate~\cite[Lem.~2.5(\rmnum{3})]{Bate-uniform-TransAMS}. Let $M$ be a reductive $k$-subgroup of $G$. Then, there is a natural inclusion $Y_k(M)\subseteq Y_k(G)$ of $k$-cocharacter groups. Let $\lambda\in Y_k(M)$. We write $P_\lambda(G)$ or just $P_\lambda$ for the $R$-parabolic subgroup of $G$ corresponding to $\lambda$, and $P_\lambda(M)$ for the $R$-parabolic subgroup of $M$ corresponding to $\lambda$. It is clear that $P_\lambda(M) = P_\lambda(G)\cap M$ and $R_u(P_\lambda(M)) = R_u(P_\lambda(G))\cap M$. If $G$ is connected, $R$-parabolic subgroups and $R$-Levi subgroups are parabolic subgroups and Levi subgroups in the usual sense~\cite[Prop.~8.4.5]{Springer-book}.

The next result is used repeatedly to reduce problems on $G$-complete reducibility to those on $L$-complete reducibility where $L$ is an $R$-Levi subgroup of $G$. 

\begin{prop}\label{G-cr-L-cr}
Suppose that a subgroup $H$ of $G$ is contained in a $k$-defined $R$-Levi subgroup $L$ of $G$. Then $H$ is $G$-cr over $k$ if and only if it is $L$-cr over $k$. 
\end{prop}
\begin{proof}
This follows from Proposition~\ref{geometric} and~\cite[Thm.~5.4(\rmnum{2})]{Bate-cocharacter-Arx}. 
\end{proof}

The next result shows how complete reducibility behaves under central isogenies. 
\begin{defn}
Let $G_1$ and $G_2$ be reductive $k$-groups. A $k$-isogeny $f:G_1\rightarrow G_2$ is \emph{central} if $\textup{ker}\,df_1$ is central in $\mathfrak{g_1}$ where $\textup{ker}\,df_1$ is the differential of $f$ at the identity of $G_1$ and $\mathfrak{g_1}$ is the Lie algebra of $G_1$. 
\end{defn}
\begin{prop}\label{isogeny}
Let $G_1$ and $G_2$ be reductive $k$-groups. Let $f:G_1 \rightarrow G_2$ be a central $k$-isogeny. 
\begin{enumerate}
\item{Suppose that $H_1<G_1^{\circ}$ (this holds, in particular, if $H_1$ is connected). If $H_1$ is $G_1$-cr over $k$, then $f(H_1)$ is $G_2$-cr over $k$.}
\item{Suppose that $f^{-1}(H_2)<G_1^\circ$. If $H_2$ is $G_2$-cr over $k$, then $f^{-1}(H_2)$ is $G_1$-cr over $k$.} 
\end{enumerate}
\end{prop}
\begin{proof}
Proposition~\ref{geometric} and~\cite[Cor.~5.3]{Bate-cocharacter-Arx} show that a subgroup $H<G^{\circ}$ of a reductive $G$ is $G$-cr over $k$ if and only if it is $G^{\circ}$-cr over $k$. Now the result follows from the connected case~\cite[Prop.~1.12]{Uchiyama-Nonperfect-pre}.
\end{proof}
\begin{rem}
In Proposition~\ref{isogeny} if we know that a $k$-defined $R$-parabolic subgroup of $G_1$ always arises as the inverse image of a $k$-defined $R$-parabolic subgroup of $G_2$, then a similar argument as in the proof of~\cite[Prop.~1.12]{Uchiyama-Nonperfect-pre} goes through and we can omit the assumptions ``$H_1<G_1^{\circ}$'' in Part 1 and ``$f^{-1}(H_2)<G_1^\circ$'' in Part 2. We do not know if this is the case or not.  
\end{rem}

Now we recall some terminology from GIT~\cite[Def.~1.1, Sec.~2.4]{Bate-cocharacter-Arx}. Let $V$ be a $G$-variety. Let $v\in V$. 
\begin{defn}\label{cocharacterclosure}
We say that $G(k)\cdot v$ is \emph{cocharacter closed over $k$} if for every $\lambda\in Y_k(G)$ such that $v':=\text{lim}_{a\rightarrow 0}\lambda(a)\cdot v$ exists, $v'$ is $G(k)$-conjugate to $v$. Moreover, we say that $G\cdot v$ is \emph{cocharacter closed} if for every cocharacter $\lambda$ of $G$ such that $v':=\text{lim}_{a\rightarrow 0}\lambda(a)\cdot v$ exists, $v'$ is $G$-conjugate to $v$. 
\end{defn}
Note that by the Hilbert-Mumford theorem~\cite{Kempf-instability-Ann}, $G\cdot v$ is cocharacter closed if and only if it is Zariski closed. 

\begin{defn}\label{destabilizingcocharacter}
Let $\lambda\in Y_k(G)$. We say that \emph{$\lambda$ destabilizes $v$ over $k$} if $\text{lim}_{a\rightarrow 0}\lambda(a)\cdot v$ exists. Moreover if $v':=\text{lim}_{a\rightarrow 0}\lambda(a)\cdot v$ exists and $v'$ is not $G(k)$-conjugate to $v$, we say that \emph{$\lambda$ properly destabilizes $v$ over $k$}. Similarly, for $\lambda\in Y(G)$, if $\text{lim}_{a\rightarrow 0}\lambda(a)\cdot v$ exists, we say that \emph{$\lambda$ destabilizes $v$}. If $v':=\text{lim}_{a\rightarrow 0}\lambda(a)\cdot v$ exists for $\lambda\in Y(G)$ and $v'$ is not $G$-conjugate to $v$, we say that \emph{$\lambda$ properly destabilizes $v$}. 
\end{defn}

We use the following very useful results from GIT~\cite[Thm.~3.3]{Bate-uniform-TransAMS} and~\cite[Cor.~5.1]{Bate-cocharacter-Arx}.
\begin{prop}\label{perfectconjugacy}
Let $k$ be perfect. Suppose that $v':=\text{lim}_{a\rightarrow 0}\lambda(a)\cdot v$ exists for $\lambda\in G$ and $v'$ is $G(k)$-conjugate to $v$. Then $v'$ is $R_u(P_\lambda)(k)$-conjugate to $v$. 
\end{prop}   
For nonperfect $k$, we do not know whether Proposition~\ref{perfectconjugacy} still holds~\cite[Question 7.8]{Bate-cocharacter-Arx}. 
It is known that if the centralizer of $v$ in $G$ is separable, it holds for nonperfect $k$~\cite[Thm.~7.1]{Bate-cocharacter-Arx}. See~\cite{Bate-cocharacter-Arx} and~\cite{Bate-cocharacterbuildings-Arx} for details. 

\begin{prop}\label{closureconjugacy}
Suppose that $v':=\text{lim}_{a\rightarrow 0}\lambda(a)\cdot v$ exists for $\lambda\in G$ and $v'$ is $G(k)$-conjugate to $v$.
If $G(k)\cdot v$ is cocharacter closed over $k$, then $v'$ is $R_u(P_\lambda)(k)$-conjugate to $v$. 
\end{prop}   

\section{$G$-cr over $k$ vs $G$-cr}
We now prove Theorems~\ref{Prop1} and~\ref{Prop2}. Our proof works for both connected and non-connected $G$ in a uniform way.

\begin{proof}[Proof of Theorem~\ref{Prop1}]
Since $H$ is not $G$-cr over $k$, there exists a proper $k$-defined $R$-parabolic subgroup $P$ of $G$ containing $H$. Let $P=P_\lambda$ be a minimal such $k$-defined $R$-parabolic subgroup where $\lambda\in Y_k(G)$. Since $H$ is $G$-cr and $R$-Levi subgroups of $P_\lambda$ are $R_u(P_\lambda)(\overline k)$-conjugate by~\cite[Cor.~6.7]{Bate-geometric-Inventione}, there exists $u\in R_u(P_\lambda)(\overline k)$ such that $H$ is contained in $u^{-1}\cdot L_{\lambda}$. Then $u\cdot H$ is contained in $L_\lambda$. Suppose that $u\cdot H$ is not $G$-cr over $k$. Then it is not $L_\lambda$-cr over $k$ by Proposition~\ref{G-cr-L-cr}. So there exists a proper $k$-defined $R$-parabolic subgroup $P_L$ of $L_\lambda$ containing $u\cdot H$. Thus $u\cdot H$ is contained in a $k$-defined $R$-parabolic subgroup $Q:=P_L\ltimes R_u(P_\lambda)$ of $G$ (note: the subgroup $Q$ is an $R$-parabolic subgroup of $G$ by~\cite[Lem.~6.2(\rmnum{2})]{Bate-geometric-Inventione}). Then $H$ is contained in $u^{-1}\cdot Q$. Note that $u\in R_u(P_\lambda)<Q$. Thus $u^{-1}\cdot Q=Q$ and we have $H<Q$. It is clear that $Q$ is strictly contained in $P_\lambda$. This contradicts the minimality of $P_\lambda$. So we conclude that $u\cdot H$ is $G$-cr over $k$.   
\end{proof}

\begin{proof}[Proof of Theorem~\ref{Prop2}]
We start with Part 1. Let $P_\lambda$ be minimal among $k$-defined $R$-parabolic subgroups of $G$ containing $H$. Since $H$ is $G$-cr over $k$, there exists a $k$-defined $R$-Levi subgroup $L$ of $P_\lambda$ containing $H$. Since $k$-defined $R$-Levi subgroups of $P_\lambda$ are $R_u(P_\lambda)(k)$-conjugate, there exists $u\in R_u(P_\lambda)(k)$ such that $L=u\cdot L_\lambda$. Set $L_\mu:=u\cdot L_\lambda$. Suppose that $H$ is not $L_\mu$-ir over $k$. So there exists a $k$-defined proper $R$-parabolic subgroup $P_L$ of $L_\mu$ containing $H$. Then we have $H<P_L<P_L\ltimes R_u(P_\lambda)$.
Since $P_L\ltimes R_u(P_\lambda)$ is a $k$-defined $R$-parabolic subgroup of $G$ strictly contained in $P_\lambda$, this contradicts the minimality of $P_\lambda$.

For part 2, let $L$ be a minimal $k$-defined $R$-Levi subgroup containing $H$. Since $H$ is not $L$-cr, there exists a proper $R$-parabolic subgroup of $L$ containing $H$. Let $P_L$ be a minimal such $R$-parabolic subgroup of $L$. Since an $R$-parabolic subgroup of $L$ is $L(\overline k)$-conjugate to a $k$-defined $R$-parabolic subgroup of $L$, there exists $l\in L(\overline k)$ such that $l\cdot P_L$ is a $k$-defined $R$-parabolic subgroup of $L$. Then $l\cdot H < l\cdot P_L$. Suppose that $l\cdot H$ is $G$-cr over $k$. Then $l\cdot H$ is $L$-cr over $k$ by Proposition~\ref{G-cr-L-cr}, so there exists a $k$-defined $R$-Levi subgroup $M$ of $l\cdot P_L$ containing $l\cdot H$. Note that $l\cdot H$ is not $M$-cr since $H$ is not $G$-cr. Then there exists a proper $R$-parabolic subgroup $P_M$ of $M$ containing $l\cdot H$. Thus $P':=P_M\ltimes R_u(l\cdot P_L)$ is an $R$-parabolic subgroup of $L$ containing $l\cdot H$. Then $l^{-1}\cdot P'$ is an $R$-parabolic subgroup of $L$ containing $H$. It is clear that $l^{-1}\cdot P'$ is a proper subgroup of $P_L$. This contradicts the minimality of $P_L$. Thus $l\cdot H$ is not $G$-cr over $k$.  
\end{proof}

\begin{rem}
Although Theorems~\ref{Prop1} and~\ref{Prop2} (and ideas in the proofs) explain necessary conditions to have examples of a subgroup $H$ of $G$ that is $G$-cr over $k$ but not $G$-cr (or vice versa), it is still a difficult problem to find concrete such examples with a $k$-defined $H$. In the next section, we use the converse of Theorems~\ref{Prop1} and~\ref{Prop2}: we start with some subgroup of $G$ and conjugate it by $u$ (or $l$) as in the proof of Theorems~\ref{Prop1} and~\ref{Prop2} to obtain a subgroup with the desired property. For our construction to work, $u$ (or $l$) needs to be chosen very carefully and the choice is closely related to the nonseparability of $H$. We show all details in the next section. The same idea was used in~\cite{Bate-separability-TransAMS},~\cite{Uchiyama-Nonperfect-pre},~\cite{Uchiyama-Classification-pre}, and~\cite{Uchiyama-Separability-JAlgebra}.  
\end{rem}

\section{The $D_4$ example}
In this section we prove Theorem~\ref{D4example}. We use the triality of $D_4$ in an essential way.\\

Let $\tilde{G}$ be a simple algebraic group of type $D_4$ defined over a nonperfect field of characteristic $2$. Fix a maximal $k$-torus of $\tilde{G}$ and a $k$-defined Borel subgroup of $\tilde{G}$. Let $\Psi(\tilde G)=\Psi(\tilde G,T)$ be the set of roots corresponding to $T$, and $\Psi(\tilde{G})^{+}=\Psi(\tilde{G},B,T)$ be the set of positive roots of $\tilde{G}$ corresponding to $T$ and $B$. Label the simple roots of $\tilde{G}$ as depicted in the following Dynkin diagram.
\begin{figure}[h]
                \centering
                \scalebox{0.7}{
                \begin{tikzpicture}
                \draw (1,0) to (2,0);
                \draw (2,0) to (2.5,0.85);
                \draw (2,0) to (2.5,-0.85);
                \fill (1,0) circle (1mm);
                \fill (2,0) circle (1mm);
                \fill (2.5,0.85) circle (1mm);
                \fill (2.5,-0.85) circle (1mm);
                \draw[below] (1,-0.2) node{$\alpha$};
                \draw[below] (1.8,-0.2) node{$\beta$};
                \draw[below] (2.8,1.1) node{$\gamma$};
                \draw[below] (2.8,-0.7) node{$\delta$};
                \draw [->,thick] (1.0,0.35) to [out=70,in=160] (2.0,0.9);
                \node (a) at (1.5, 1.1) {$\sigma$};
                \draw [->,thick] (2.9,0.6) to [out=-30,in=30] (2.9,-0.6);
                \draw [->,thick] (2.1,-1.1) to [out=180,in=-60] (1.0,-0.7);
               \end{tikzpicture}}
\end{figure}
Let $G:=\tilde{G}\rtimes \langle \sigma \rangle$ where $\sigma$ is the non-trivial element of the graph automorphism group of $\tilde{G}$ (normalizing $T$ and $B$) as depicted in the diagram; we have $\sigma\cdot \alpha=\gamma, \; \sigma\cdot \gamma=\delta, \; \sigma\cdot \delta=\alpha$, and $\beta$ is fixed by $\sigma$. We label $\Psi(\tilde{G})^{+}$ as follows. The corresponding negative roots are labeled accordingly. Note that Roots 1, 2, 3, 4 correspond to $\alpha$, $\gamma$, $\delta$, $\beta$ respectively.

\begin{table}[!h]
\begin{center}
\scalebox{0.7}{
\begin{tabular}{cccccc}
\rootsD{1}{0}{1}{0}{0}&\rootsD{2}{1}{0}{0}{0}&\rootsD{3}{0}{0}{0}{1}&\rootsD{4}{0}{0}{1}{0}&\rootsD{5}{0}{1}{1}{0}&\rootsD{6}{1}{0}{1}{0}\\
\rootsD{7}{0}{0}{1}{1}&\rootsD{8}{1}{1}{1}{0}&\rootsD{9}{0}{1}{1}{1}&\rootsD{10}{1}{0}{1}{1}&\rootsD{11}{1}{1}{1}{1}&\rootsD{12}{1}{1}{2}{1}\\
\end{tabular}
}
\end{center}
\end{table}   

Define
$\lambda:=(\alpha+2\beta+\gamma+\delta)^{\vee}=\alpha^{\vee}+2\beta^{\vee}+\gamma^{\vee}+\delta^{\vee}$. 
Then 
\begin{alignat*}{2}
P_\lambda&=\langle T, \sigma, U_{\zeta}\mid \zeta\in \Psi(\tilde G)^{+}\cup \{-1,-2,-3\} \rangle,\\
L_\lambda&=\langle T, \sigma, U_{\zeta}\mid \zeta\in \{\pm 1,\pm 2,\pm 3\} \rangle,\\
R_u(P_\lambda)&=\langle U_{\zeta} \mid \zeta \in \Psi(\tilde G)^{+}\backslash \{1, 2, 3\} \rangle.
\end{alignat*}
Let $a\in k\backslash k^2$. Let $v(\sqrt a):=\epsilon_{6}(\sqrt a)\epsilon_{9}(\sqrt a)\in R_u(P_\lambda)(\overline k)$. Define
\begin{equation*}
H:=v(\sqrt a)\cdot\langle (n_{\alpha}\sigma), \; (\alpha+\gamma)^{\vee}(\overline k^*) \rangle.
\end{equation*}
Here is our first main result in this section.
\begin{prop}\label{firstmain}
$H$ is $k$-defined. Moreover, $H$ is $G$-cr but not $G$-cr over $k$. 
\end{prop}
\begin{proof}
First, we have 
$
(n_\alpha \sigma) \cdot (\beta+\delta) = (n_\alpha \sigma) \cdot 6 = 9, \;
(n_\alpha \sigma) \cdot 9 = 6.
$
Using this and the commutation relations~\cite[Lem.~32.5 and Prop.~33.3]{Humphreys-book1}, we obtain
\begin{equation*}
v(\sqrt a)\cdot (n_{\alpha}\sigma)=(n_\alpha \sigma) \epsilon_{12}(a).
\end{equation*}
An easy computation shows that $v(\sqrt a)$ commutes with $(\alpha+\gamma)^{\vee}(\overline k^*)$. Now it is clear that $H$ is $k$-defined. 

Now we show that $H$ is $G$-cr. It is sufficient to show that $H':=v(\sqrt a)^{-1}\cdot H=\langle n_\alpha \sigma,\; (\alpha+\gamma)^{\vee}(\overline k^*)\rangle$ is $G$-cr since it is $G$-conjugate to $H$. Since $H'$ is contained in $L_\lambda$, by Proposition~\ref{G-cr-L-cr} it is enough to show that $H'$ is $L_\lambda$-cr. We actually show that $H'$ is $L_\lambda$-ir. Note that $L_\lambda=A_1\times A_1 \times A_1=L_\alpha\times L_\gamma\times L_\delta$. We have 
\begin{equation*}
(n_\alpha \sigma)\cdot (\alpha+\gamma)^{\vee}(\overline k^*) = (\gamma+\delta)^{\vee}(\overline k^*),\;
(n_\alpha \sigma)^3 = n_\alpha n_\gamma n_\delta.
\end{equation*}  
Thus $H'$ contains $(\alpha+\gamma)^{\vee}(\overline k^*)$, $(\gamma+\delta)^{\vee}(\overline k^*)$, and $n_\alpha n_\gamma n_\delta$. Now it is clear that $H'$ is $L_\lambda$-ir. 

Next, we show that $H$ is not $G$-cr over $k$. Suppose the contrary. Clearly $H$ is contained in a $k$-defined $R$-parabolic subgroup $P_\lambda$. Then there exists a $k$-defined $R$-Levi subgroup of $P_\lambda$ containing $H$. Then by~\cite[Lem.~2.5(\rmnum{3})]{Bate-uniform-TransAMS} there exists $u\in R_u(P_\lambda)(k)$ such that $H$ is contained in $u\cdot L_\lambda$. Thus $n_\alpha\sigma\epsilon_{12}(a) < u\cdot L_\lambda$. So $u^{-1}\cdot (n_\alpha\sigma \epsilon_{12}(a)) < L_{\lambda}$. By~\cite[Prop.~8.2.1]{Springer-book}, we can write
\begin{equation*}
u=\prod_{\zeta\in \Psi(R_u(P_\lambda))}\epsilon_\zeta(x_\zeta), 
\end{equation*}
for some scalars $x_\zeta$. We compute how $n_\alpha \sigma$ acts $\Psi(R_u(P_\lambda))$. Using the labelling of the positive roots above, we have $\Psi(R_u(P_\lambda))=\{4,\ldots, 12\}$. We compute how $n_\alpha \sigma$ acts on $\Psi(R_u(P_\lambda))$: 
\begin{equation}\label{perm}
n_\alpha \sigma = (4\;5\;8\;11\;10\;7) (6\;9) (12). 
\end{equation}
Using this and the commutation relations,
\begin{alignat*}{2}
u^{-1}\cdot (n_\alpha\sigma \epsilon_{12}(a))
=\; &n_\alpha \sigma\epsilon_7(x_4+x_7)\epsilon_{10}(x_7+x_{10})\epsilon_{9}(x_6+x_9)\epsilon_{11}(x_{10}+x_{11})\\
&\epsilon_{6}(x_6+x_9)\epsilon_{8}(x_8+x_{11})\epsilon_{4}(x_4+x_5)\epsilon_{5}(x_5+x_8)\\
&\epsilon_{12}(x_5 x_{10}+x_{5}x_{11}+x_7 x_8 +x_7 x_{11}+x_8 x_{10}+{x_9}^2+a).
\end{alignat*}
Thus if $u^{-1}\cdot (n_\alpha\sigma \epsilon_{\alpha+2\beta+\gamma+\delta}(a)) < L_{\lambda}$ we must have
\begin{alignat*}{2}
&x_4=x_5=x_7=x_{8}=x_{10}=x_{11},\; x_6=x_9,\\
&x_5 x_{10}+x_{5}x_{11}+x_7 x_8 +x_7 x_{11}+x_8 x_{10}+{x_9}^2+a=0.
\end{alignat*}
Set $x_4=y$. Then we have $y^2+{x_9}^2+a=0$. Thus $(y+x_9)^2=a$. This is impossible since $y, x_9\in k$ and $a\notin k^2$. We are done. 
\end{proof}

\begin{rem}\label{D4nonsep}
From the computations above we see that the curve $C(x):=\{\epsilon_{6}(x)\epsilon_{9}(x)\mid x\in \overline k\}$ is not contained in $C_G(H)$, but the corresponding element in $\textup{Lie}(G)$, that is, $e_6+e_9$ is contained in $\mathfrak{c}_{\mathfrak{g}}(H)$. Then the argument in the proof of~\cite[Prop.~3.3]{Uchiyama-Separability-JAlgebra} shows that $\textup{Dim}(C_G(H))$ is strictly smaller than $\textup{Dim}(\mathfrak{c}_{\mathfrak{g}}(H))$. So $H$ is non-separable in $G$. 
\end{rem} 

\vspace{5mm}
Now we move on to the second main result in this section. We use the same $G$, $\lambda$, and $a$ as above. We also use the same labelling of the roots of $G$. Let $v(\sqrt a):=\epsilon_{-6}(\sqrt a)\epsilon_{-9}(\sqrt a)$. Let
\begin{equation*}
K:=v(\sqrt a)\cdot \langle n_{\alpha}\sigma,\; (\alpha+\gamma)^{\vee}(\overline k^*)\rangle=\langle n_\alpha \sigma \epsilon_{-12}(a), \;  (\alpha+\gamma)^{\vee}(\overline k^*)\rangle. 
\end{equation*}
Define
\begin{equation*}
H:=\langle K, \; \epsilon_{11}(1) \rangle.
\end{equation*}

\begin{prop}\label{secondmain}
$H$ is $k$-defined. Moreover, $H$ is $G$-ir over $k$ but not $G$-cr. 
\end{prop}
It is clear that $H$ is $k$-defined. We now prove Proposition~\ref{secondmain} with a series of lemmas. Define $M:=\langle n_\alpha \sigma, (\alpha+\gamma)^{\vee}(\overline k^*)\rangle$.
\begin{lem}\label{centralizerofM}
$C_G(M)^{\circ}=G_{12}$.
\end{lem}
\begin{proof}
First of all, from Equation~(\ref{perm}) we see that $G_{12}$ is contained in $C_G(n_\alpha \sigma)$. Since $\langle \alpha+2\beta+\gamma+\delta, (\alpha+\gamma)^{\vee} \rangle=0$, $G_{12}$ is also contained in $C_G((\alpha+\gamma)^{\vee}(\overline k^*))$. So $G_{12}$ is contained in $C_G(M)$. Recall that by~\cite[Thm.~13.4.2]{Springer-book} the image of the product map $C_{R_u(P_\lambda)}(M)^{\circ}\times C_{L_\lambda}(M)^{\circ}\times C_{R_u(P_\lambda^{-})}(M)^{\circ}\rightarrow G$ is an open subset of $C_G(M)^{\circ}$ where $P_\lambda^{-}$ is the opposite of $P_\lambda$ containing $L_\lambda$. Set $u:=\prod_{i\in \Psi(R_u(P_\lambda))}\epsilon_{i}(x_i)$ for some $x_i \in \overline k$. Using Equation~(\ref{perm}) and the commutation relations, we obtain
\begin{alignat*}{2}
(n_\alpha \sigma)\cdot u &= \epsilon_{4}(x_7)\epsilon_{5}(x_4)\epsilon_{6}(x_9)\epsilon_7(x_{10})\epsilon_8(x_{5})\epsilon_9(x_6)\epsilon_{10}(x_{11})\epsilon_{11}(x_8)\epsilon_{12}(x_5 x_{10}+x_6 x_9 +x_{12}). 
\end{alignat*}
So, if $u\in C_{R_u(P_\lambda)}(n_\alpha \sigma)$ then
$x_4=x_5=x_7=x_{8}=x_{10}=x_{11},\; x_6=x_9$. But $\langle (\alpha+\gamma)^{\vee}, \alpha+\beta \rangle =2$, so $x_5=0$ for $u\in C_{R_u(P_\lambda)}(M)$. Then 
\begin{equation*}
(n_\alpha \sigma)\cdot u = \epsilon_6(x_6)\epsilon_9(x_6)\epsilon_{12}(x_6^2+x_{12}).
\end{equation*}
So we must have $x_6^2=0$ if $u\in C_{R_u(P_\lambda)}(M)$. Thus we conclude that $C_{R_u(P_\lambda)}(M)=U_{12}$. Similarly, we can show that $C_{R_u(P_{\lambda}^{-})}(M)=U_{-12}$. Now we show that $C_L(M)<G_{12}$. In the proof of Proposition~\ref{firstmain} we have shown that $(\gamma+\delta)^{\vee}(\overline k^*)$ is contained in $M$. So $C_{L_\lambda}(M)$ is contained in $C_{L_\lambda}((\alpha+\gamma)^{\vee}(\overline k^*),\; (\gamma+\delta)^{\vee}(\overline k^*))$. A direct computation shows that $C_{L_\lambda}((\alpha+\gamma)^{\vee}(\overline k^*),\; (\gamma+\delta)^{\vee}(\overline k^*)) = T$ and $C_T(n_\alpha \sigma)=(\alpha+2\beta+\gamma+\delta)^{\vee}(\overline k^*)<G_{12}$. We are done.
\end{proof}

Note that
\begin{equation*}
v(\sqrt a)^{-1}\cdot H = \langle n_\alpha \sigma, \; (\alpha+\gamma)^{\vee}(\overline k^*),\; \epsilon_{11}(1)\epsilon_{2}(\sqrt a)\rangle.
\end{equation*}
Thus we see that $v(\sqrt a)^{-1}\cdot H$ is contained in $P_\lambda$. So $H$ is contained in $v(\sqrt a)\cdot P_\lambda$.

\begin{lem}\label{uniquepara}
$v(\sqrt a)\cdot P_\lambda$ is the unique proper $R$-parabolic subgroup of $G$ containing $H$.
\end{lem}
\begin{proof}
Suppose that $P_\mu$ is a proper $R$-parabolic subgroup containing $v(\sqrt a)^{-1}\cdot H$. We show that $P_\mu = P_\lambda$. In the proof of Proposition~\ref{firstmain} we have shown that $M$ is $G$-cr. Then there exists an $R$-Levi subgroup $L$ of $P_\mu$ containing $M$ since $M$ is contained in $P_\mu$. Since $R$-Levi subgroups of $P_\mu$ are $R_u(P_\mu)$-conjugate by~\cite[Lem.~2.5(\rmnum{3})]{Bate-uniform-TransAMS}, without loss, we can assume $L=L_\mu$. Then $M<L_\mu=C_G(\mu(\overline k^*))$, so $\mu(\overline k^*)$ centralizes $M$.

Since $\lambda(\overline k^*)$ centralizes $M$, Lemma~\ref{centralizerofM} yields $\mu(\overline k^*)<G_{12}$. Then we can take $\mu=g\cdot (\alpha+2\beta+\gamma+\delta)^{\vee}$ for some $g\in G_{12}$. (This does not change $P_\mu$ or $L_\mu$.) By the Bruhat decomposition, $g$ has one of the following forms:
\begin{align}
g &=(\alpha+2\beta+\gamma+\delta)^{\vee}(s)\epsilon_{12}(x_1), \label{eqn1}\\
g &=\epsilon_{12}(x_1)n_{\alpha+2\beta+\gamma+\delta}(\alpha+2\beta+\gamma+\delta)^{\vee}(s)\epsilon_{12}(x_2) \label{eqn2}
\end{align}
for some $x_1$, $x_2 \in \overline k$, $s\in \overline k^*$. %
We rule out the second case. Suppose $g$ is of the second form. Note that $\epsilon_{11}(1)\epsilon_{2}(\sqrt a)\in v(\sqrt a)^{-1}\cdot H< P_\mu$. 
But $P_\mu=P_{g\cdot (\alpha+2\beta+\gamma+\delta)^{\vee}}=g\cdot P_{(\alpha+2\beta+\gamma+\delta)^{\vee}}$. So it is enough to show that $g^{-1}\cdot (\epsilon_{11}(1)\epsilon_{2}(\sqrt a))\notin P_{(\alpha+2\beta+\gamma+\delta)^{\vee}}$. Since $U_{12}$ and $(\alpha+2\beta+\gamma+\delta)(\overline k^*)$ are contained in $P_{(\alpha+2\beta+\gamma+\delta)^{\vee}}$ we can assume $g=n_{12}$. We have
\begin{equation*}
n_{12}=n_\alpha n_\beta n_\alpha n_\gamma n_\beta n_\alpha n_\delta n_\beta n_\alpha n_\gamma n_\beta n_\delta \textup{ (the longest element in the Weyl group of $D_4$)}.
\end{equation*}
Using this, we can compute how $n_{12}$ acts on each root subgroup of $G$. In particular $n_{12}^{-1}\cdot U_{11}=U_{-12}$ and $n_{12}^{-1}\cdot U_{2}= U_{-2}$. Thus
\begin{alignat*}{2}
n_{12}^{-1}\cdot (\epsilon_{11}(1)\epsilon_{2}(\sqrt a)) &= \epsilon_{-12}(1)\epsilon_{-2}(\sqrt a)\notin P_{(\alpha+2\beta+\gamma+\delta)^{\vee}}.
\end{alignat*}
So $g$ must be of the first form as claimed. Then $g\in P_\lambda$. Thus $P_\mu=P_{g\cdot \lambda}=g\cdot P_\lambda=P_\lambda$. We are done.
\end{proof}

\begin{lem}\label{nonkdefined}
$v(\sqrt a)\cdot P_\lambda$ is not $k$-defined.
\end{lem}
\begin{proof}
Suppose the contrary. Then $(v(\sqrt a)\cdot P_\lambda)^{\circ}$ is $k$-defined. Since $P_\lambda^{\circ}$ is $k$-defined, $v(\sqrt a)\cdot P_\lambda$ is $G^{\circ}(k)$-conjugate to $P_\lambda^{\circ}$ by~\cite[Thm.~20.9]{Borel-AG-book}. Thus $g v(\sqrt a)\cdot P_\lambda^{\circ}$ for some $g\in G(k)^{\circ}$. So $g v(\sqrt a)\in P_\lambda^{\circ}$ since parabolic subgroups are self-normalizing. Then $g=pv(\sqrt a)^{-1}$ for some $p\in P_\lambda^{\circ}$. Thus $g$ is a $k$-point of $P_\lambda^{\circ} R_u(P_\lambda^{-})$. Then by the rational version of the Bruhat decomposition~\cite[Thm.~21.15]{Borel-AG-book}, there exists a unique $p'\in P_\lambda^{\circ}(k)$ and a unique $u'\in R_u(P_\lambda^{-})(k)$ such that $g=p' u'$. This is a contradiction since $v(\sqrt a)\notin R_u(P_\lambda^{-})(k)$. 
\end{proof}

\begin{lem}\label{nonG-cr}
$H$ is not $G$-cr. 
\end{lem}
\begin{proof}
We have shown that $C_G(M)^{\circ}=G_{12}$. Then $C_G(v(\sqrt a)^{-1}\cdot H)^{\circ}<G_{12}$ since $M<v(\sqrt a)^{-1}\cdot H$. Using the commutation relations, we see that $U_{12}< C_G(v(\sqrt a)^{-1}\cdot H)$. Note that $v(\sqrt a)^{-1}\cdot H$ contains $h:=\epsilon_{11}(1)\epsilon_{2}(\sqrt a)$ which does not commute with any non-trivial element of $U_{-12}$. Also, since $\langle \alpha+\beta+\gamma+\delta, \lambda\rangle = 4$,  $h$ does not commute with any non-trivial element of $(\alpha+2\beta+\gamma+\delta)^{\vee}(\overline k^*)$.
Thus we conclude that $C_G(v(\sqrt a)^{-1}\cdot H)^{\circ}=U_{12}$. So $C_G(H)^{\circ}=v(\sqrt a)\cdot U_{12}$ which is unipotent. Then by the classical result of Borel and Tits~\cite[Prop.~3.1]{Borel-Tits-unipotent-invent}, we see that $C_G(H)^{\circ}$ is not $G$-cr.
Since $C_G(H)^{\circ}$ is a normal subgroup of $C_G(H)$, by~\cite[Ex.~5.20]{Bate-uniform-TransAMS}, $C_G(H)$ is not $G$-cr. Then by~\cite[Cor.~3.17]{Bate-geometric-Inventione}, $H$ is not $G$-cr. 
\end{proof}

Now Lemmas~\ref{uniquepara},~\ref{nonkdefined} show that $H$ is $G$-ir over $k$, and Lemma~\ref{nonG-cr} shows that $H$ is not $G$-cr. This proves Proposition~\ref{secondmain}. 

\begin{rem}
Now we have a collection of examples of subgroups of $G$ that are $G$-cr over $k$ but not $G$-cr (or vice versa) for connected $G$ of type $G_2$, $E_6$, $E_7$, $E_8$, and for non-connected $G$ of type $A_2$, $A_4$, $D_4$ (\cite{Bate-separability-TransAMS},~\cite{Uchiyama-Nonperfect-pre},~\cite{Uchiyama-Classification-pre}, and~\cite{Uchiyama-Separability-JAlgebra}) all in characteristic $2$. It would be interesting to find such examples in characteristic $3$ (or any other bad characteristic).   
\end{rem}

In general, combining~\cite[Thm.~1.5]{Bate-cocharacter-Arx} and Proposition~\ref{geometric} we have
\begin{prop}\label{separableG-cr}
Let $k$ be nonperfect. Suppose that $H$ is a separable $k$-subgroup of $G$. If $H$ is $G$-cr, then it is $G$-cr over $k$.
\end{prop}
Our examples of subgroups $H$ of $G$ in~\cite{Uchiyama-Nonperfect-pre} (and the $D_4$ example in this paper) for the other direction ($G$-cr over $k$ but not $G$-cr) are all nonseparable. So it is natural to conjecture that the other direction of Proposition~\ref{separableG-cr} holds if $H$ is separable. However, there exists a separable such subgroup in $G=PGL_2$: 
\begin{example}\label{PGL-nonplongeable}
Let $k$ be a nonperfect field of characteristic $p=2$. Let $a\in k\backslash k^2$. Let $G=PGL_2$. We write $\bar A$ for the image in $PGL_2$ of $A\in GL_2$. Set $u= \overline{\left[\begin{array}{cc}
                        0 & a \\
                        1 & 0 \\
                         \end{array}  
                        \right]}\in G(k)$. Let $U:=\langle u \rangle$. Then $U$ is unipotent, so by the classical result of Borel-Tits~\cite[Prop.~3.1]{Borel-Tits-unipotent-invent} $U$ is not $G$-cr. However $U$ is not contained in any proper $k$-parabolic subgroup of $G$ since there is no nontrivial $k$-defined flag of $\mathbb{P}^1_k$ stabilized by $U$. So $U$ is $G$-ir over $k$, hence $G$-cr over $k$. Also the argument in~\cite[Ex.~7.6]{Bate-cocharacter-Arx} shows that $U$ is separable in $G$. 
\end{example}
The element $u$ above is an example of a \emph{$k$-nonplongeable unipotent element}~\cite{Tits-unipotent-Utrecht}. 
\begin{defn}
A unipotent element $u$ of $G$ is \emph{$k$-nonplongeable unipotent} if $u$ is not contained in the unipotent radical of any $k$-defined $R$-parabolic subgroup of $G$. In particular, if $u$ is not contained in any $k$-defined $R$-parabolic subgroup of $G$, $u$ is \emph{$k$-anisotropic unipotent}. 
\end{defn}
Note that our definition of $k$-nonplongeability (and $k$-anisotropy) extends the original definition of Tits to non-connected $G$. Now we ask:
\begin{question}\label{Ascent}
Let $k$ be nonperfect and $G^{\circ}$ be simply-connected. Let $H$ be a $k$-subgroup of $G$. Suppose that every unipotent element of $G(k)$ is $k$-plongeable and $H$ is separable in $G$. Then if $H$ is $G$-cr over $k$, is it $G$-cr?
\end{question} 
\begin{rem}
The simply-connectedness assumption for $G^{\circ}$ is necessary to avoid an easy counterexample to the question: just copy Example~\ref{PGL-nonplongeable} but in $G=GL_2$ instead of $PGL_2$. Every subgroup is separable and $k$-plongeable in $GL_2$. An easy computation shows that the subgroup $U$ of $G=GL_2$ is $G$-cr over $k$ but not $G$-cr. If every unipotent element of $G^{\circ}(k)$ is $k$-plongeable (in particular this holds if $[k:k^p]\leq p$ and $G^{\circ}$ is simply-connected by a deep result of Gille in Galois cohomology~\cite{Gille-unipotent-Duke}) and if a connected $k$-subgroup $H$ of $G^{\circ}$ is $G^{\circ}$-cr over $k$, then $H$ is \emph{pseudo-reductive}~\cite[Thm.~1.9]{Uchiyama-Nonperfect-pre}; see~\cite[Def.~1.1.1]{Conrad-pred-book} for the definition of pseudo-reductivity. The proof of~\cite[Thm.~1.9]{Uchiyama-Nonperfect-pre} depends on the center conjecture. Question~\ref{Ascent} is closely related to the so-called ``strengthened Hilbert-Mumford theorem'' in GIT; see~\cite[Sec.~5]{Bate-cocharacter-Arx}. 
\end{rem}

\section{Geometric Invariant Theory}
In this section, we generalize Theorems~\ref{Prop1} and~\ref{Prop2} via GIT, and obtain new results (Theorems~\ref{GIT1} and~\ref{GIT2}) concerning $G$- and $G(k)$-orbits in an arbitrary $G$-variety. The proof of Theorems~\ref{GIT1} and~\ref{GIT2} give a new GIT-theoretic proof to Theorems~\ref{Prop1} and~\ref{Prop2} using the geometric characterization of complete reducibility (Proposition~\ref{geometric}). Let $H$ be a subgroup of $G$ such that $H=\langle h_1, \cdots, h_n \rangle$ for some $n\in \mathbb{N}$. Let $\bold h:=(h_1, \cdots, h_n)\in G^n$. Suppose that $G$ (and $G(k)$) acts on $G^n$ via simultaneous conjugation.

\begin{prop}
Suppose that $G\cdot \bold h$ is Zariski closed but $G(k)\cdot \bold h$ is not cocharacter closed over $k$. Let $P$ be a minimal $k$-defined $R$-parabolic subgroup of $G$ containing $H$. Then there exists a unipotent element $u\in R_u(P)(\overline k)$ such that $G(k)\cdot (u\cdot \bold h)=G(k)\cdot (u h_1 u^{-1},\cdots, u h_N u^{-1})$ is cocharacter closed over $k$.
\end{prop}
\begin{proof}
This is a translation of Theorem~\ref{Prop1} using Proposition~\ref{geometric}. Alternatively, this follows from~\ref{GIT1} since this is a special case of Theorem~\ref{GIT1}. 
\end{proof}
\begin{prop}
Let $\bold h$ and $H$ be as above. Suppose that $G(k)\cdot \bold h$ is cocharacter closed over $k$ but $G\cdot \bold h$ is not Zariski closed. Let $L$ be a minimal $k$-defined $R$-Levi subgroup of $G$ containing $H$. Then there exists an element $l\in L(\overline k)$ such that $G(k)\cdot (l\cdot \bold h)$ is not cocharacter closed over $k$.
\end{prop}
\begin{proof}
This is a translation of Theorem~\ref{Prop2} via Proposition~\ref{geometric}. Also this is a special case of Theorem~\ref{GIT2}.
\end{proof}

Now it is easy to see that Theorems~\ref{GIT1} and~\ref{GIT2} are natural generalizations of Theorems~\ref{Prop1} and~\ref{Prop2} via GIT. 

\begin{proof}[Proof of Theorem~\ref{GIT1}]
Let $\lambda$, $P_\lambda$, and $v$ as in the hypothesis. Let $v':=\textup{lim}_{a\rightarrow 0}\lambda(a)\cdot v$. Since $G\cdot v$ is Zariski closed, there exists $g\in G(\overline k)$ such that $v':=g\cdot v$. 
Then, by Proposition~\ref{perfectconjugacy} there exists $u\in R_u(P_\lambda)(\overline k)$ such that $v':=u\cdot v$ since $\overline k$ is perfect.  
Our goal is to show that $G(k)\cdot v'$ is cocharacter closed over $k$. Suppose the contrary. Then by Propositions~\ref{G-cr-L-cr} and~\ref{geometric}, $L_\lambda(k)\cdot v'$ is not cocharacter closed over $k$.  
Then there exists a $k$-defined cocharacter $\mu$ of $L_\lambda$ properly destabilizing $v'$. So $\mu(\overline k^*)$ is not central in $L_\lambda$. 
Then $Q:=P_\mu(L_\lambda)\ltimes R_u(P_\lambda)$ is a $k$-defined $R$-parabolic subgroup of $G$ properly contained in $P_\lambda$. Set $Q:=Q_\zeta$ for some $\zeta\in Y_k(G)$. We want to show that $\zeta$ destabilizes $v$ over $k$. By~\cite[Lem.~6.2]{Bate-geometric-Inventione}, we can set $\zeta:=m\lambda+\mu$ for some large $m\in \mathbb{N}$. Note that $\lambda$ and $\mu$ clearly commute since $\mu$ is a cocharacter of $L_\lambda=C_G(\lambda(\overline k^*))$. So, if both $\textup{lim}_{a\rightarrow 0}\lambda(a)\cdot v'$ and $\textup{lim}_{a\rightarrow 0}\mu(a)\cdot v'$ exist, we are done thanks to~\cite[Lem.~2.7]{Bate-cocharacter-Arx}. We have already shown that the second limit exists. Note that $\textup{lim}_{a\rightarrow 0}\lambda(a)\cdot u=1$ since $u\in R_u(P_\zeta)$. So we have
\begin{equation*}
\textup{lim}_{a\rightarrow 0}\lambda(a)\cdot v'=\textup{lim}_{a\rightarrow 0}\lambda(a)\cdot (u\cdot v) = \textup{lim}_{a\rightarrow 0}\lambda(a)\cdot v.
\end{equation*}
The last limit exists by our assumption, so the first limit exists. We are done. 
\end{proof}

\begin{proof}[Proof of Theorem~\ref{GIT2}]
We start with part 1. Let $\lambda$, $P_\lambda$, $v$ be as in the hypothesis. Let $v':=\textup{lim}_{a\rightarrow 0}\lambda(a)\cdot v$. Since $G(k)\cdot v$ is cocharacter closed over $k$, $v'$ is $G(k)$-conjugate to $v$. Then $v'$ is $R_u(P_\lambda)(k)$-conjugate to $v$ by Proposition~\ref{closureconjugacy} since $G(k)$ is cocharacter closed over $k$. Therefore $v'=u\cdot v$ for some $u\in R_u(P_\lambda)(k)$. Then it follows that $v$ is centralized by $u^{-1}\cdot \lambda$ by~\cite[Lem.~2.12]{Bate-uniform-TransAMS}. Let $\xi:=u^{-1}\cdot \lambda$. Then $\xi$ is a $k$-defined cocharacter of $G$ since $u$ is a $k$-point of $G$. Thus $L_\xi$ is a $k$-defined $R$-Levi subgroup of $G$. We have already shown that $\xi$ fixes $v$. It is clear that $P_\lambda=P_\xi$ since $P_\xi=P_{u^{-1}\cdot\lambda}=u^{-1}\cdot P_\lambda=P_{\lambda}$.  

For part 2, suppose that there exists a non-central $k$-defined cocharacter $\mu$ of $L_\xi$ destabilizing $v$ over $k$. Let $Q:=P_\mu(L_\xi)\ltimes R_u(P_\xi)$. By a similar argument to that in the proof of Theorem~\ref{GIT1}, we find a $k$-defined cocharacter of $L_\xi$ destabilizing $v$ over $k$ since both $\xi$ and $\lambda$ destabilize $v$ over $k$. Since $Q$ is a $k$-defined $R$-parabolic subgroup strictly contained in $P_\xi=P_\lambda$ this contradicts the minimality of $P_\lambda$. So we are done. 

For part 3, let $L:=L_\xi$. Since $G\cdot v$ is not Zariski closed, $L\cdot v$ is not Zariski closed by Propositions~\ref{G-cr-L-cr} and~\ref{geometric}. Let $\Delta_{L, v,\overline{k}}$ be the set of $\overline k$-defined cocharacters of $L$ destabilizing $v$. Then by the Hilbert-Mumford theorem~\cite{Kempf-instability-Ann}, $\Delta_{L, v,\overline{k}}$ is non-empty.
Pick $\zeta\in \Delta_{L, v,\overline{k}}$ such that $P_\zeta(L)$ is minimal among those $R$-parabolic subgroups $P_\mu(L)$ of $L$ with $\mu\in \Delta_{L, v,\overline{k}}$.
Since an $R$-parabolic subgroup of $L$ is $L(\overline k)$-conjugate to a $k$-defined $R$-parabolic subgroup of $L$, there exists $l\in L(\overline k)$ such that $l\cdot P_{\zeta}(L)$ is a $k$-defined $R$-parabolic subgroup of $L$. Then there exists a $k$-defined cocharacter $\eta$ of $L$ such that $P_\eta(L)=l\cdot P_\zeta(L)$ since $k$ is separably closed. Then $P_\eta(L)=P_{l\cdot \zeta}(L)$. Without loss of generality, we can assume that $\eta$ is $P_\eta(L)$-conjugate to $l\cdot \zeta$. So set $\eta:=s\cdot (l\cdot \zeta)$ for some $s\in P_\eta(L)$. Then $s=mu$ for some $m\in L_\eta(L)$ and $u\in R_u(P_\eta(L))$. 
Since $\eta$ fixes $m$, we have $\eta=u\cdot l \cdot \zeta$. Then
\begin{equation*}
\textup{lim}_{a\rightarrow 0}\eta(a)\cdot (l\cdot v)=\textup{lim}_{a\rightarrow 0}(u\cdot l\cdot \zeta(a))\cdot (l\cdot v)=ul \cdot (\textup{lim}_{a\rightarrow 0}\zeta(a)\cdot v).
\end{equation*}
We have assumed that the last limit exists, so the first limit also exists. Thus $\eta$ destabilizes $l\cdot v$ over $k$. Suppose that $G(k)\cdot (l\cdot v)$ is cocharacter closed over $k$. Our goal is to obtain a contradiction. Now $L(k)\cdot (l\cdot v)$ is cocharacter closed over $k$ by Propositions~\ref{G-cr-L-cr} and~\ref{geometric}. 
Then by~\cite[Lem.~2.12]{Bate-uniform-TransAMS} there exists $w\in R_u(P_\eta(L))(k)$ such that $w^{-1}\cdot \eta(\overline k^*)$ fixes $l\cdot v$. Since $C:=C_L(w^{-1}\cdot \eta(\overline k^*))$ is a $k$-defined $R$-Levi subgroup (since $w$ is a $k$-point in $L$ and $\eta$ is $k$-defined), by Propositions~\ref{G-cr-L-cr} and~\ref{geometric} again, $C(k)\cdot (l\cdot v)$ is not cocharacter closed over $k$ since $G(k)\cdot (l\cdot v)$ is not cocharacter closed over $k$. Then there exists a $k$-defined cocharacter $\tau$ of $C$ properly destabilizing $l\cdot v$ over $k$. In particular, $\tau(\overline k^*)$ does not fix $l\cdot v$. Thus $P_\tau(C)$ is a proper $k$-defined $R$-parabolic subgroup of $C$. Define $Q:=P_\tau(C)\ltimes R_u(P_\eta(L))$. Then by a similar argument to that of the proof of part 2, we find a $k$-defined cocharacter $\alpha$ of $C$ such that 
$Q=Q_\alpha$ and $\alpha$ destabilizes $l\cdot v$ over $k$ since both $\tau$ and $\eta$ destabilize $l\cdot v$ over $k$. Since $Q$ is a $k$-defined $R$-parabolic subgroup of $L$ strictly contained in $P_\eta(L)$, this contradicts the minimality of $P_L$. We are done.           
\end{proof}

\begin{rem}
Our proofs of Theorems~\ref{GIT1} and~\ref{GIT2} are very similar to those of Theorems~\ref{Prop1} and~\ref{Prop2}. We believe many results on complete reducibility (over $k$) can be generalized in a similar way. Conversely more results on GIT can be used to study complete reducibility; see~\cite{Bate-cocharacter-Arx},~\cite{Bate-uniform-TransAMS} for more on this. 
\end{rem}

\section{Tits' spherical buildings}
In this section, we translate Theorems~\ref{Prop1} and~\ref{Prop2} into the language of spherical buildings~\cite{Tits-book}. This gives a topological viewpoint to the rationality problems for complete reducibility and GIT. We assume that $G$ is connected in this section.  
\begin{defn}
We write $\Delta_k(G)$ (or $\Delta(G)$) for the spherical building corresponding to the set of proper $k$-parabolic subgroups of $G$ (or proper parabolic subgroups of $G$ respectively). It is clear that $\Delta_k(G)$ is a subbuilding of $\Delta(G)$. 
\end{defn}  

\begin{defn}
Let $H$ be a subgroup of $G$, we write $\Delta_k(G)^H$ (or $\Delta(G)^H$) for the subcomplex of $\Delta_k(G)$ (or $\Delta(G)$) corresponding to the set of proper $k$-parabolic subgroups (or proper parabolic subgroups) of $G$ containing $H$.
We call $\Delta_k(G)^H$ (or $\Delta(G)^H$) the fixed point subcomplex of $\Delta_k(G)$ (or $\Delta(G)$) with respect to $H$.
\end{defn}

We recall Serre's characterization of complete reducibility (and complete reducibility over $k$) in terms of a topological property of the corresponding fixed point subcomplexes of $\Delta(G)$ (and $\Delta_k(G)$)~\cite[Sec.~2.2]{Serre-building}:

\begin{prop}\label{serre-G-cr-builidng}
Let $H$ be a subgroup of $G$. Then 
\begin{enumerate}
\item{$H$ is $G$-cr over $k$ if and only if $\Delta_k(G)^H$ is not contractible.}
\item{$H$ is $G$-cr if and only if $\Delta(G)^H$ is not contractible.}
\end{enumerate}
\end{prop} 

We are ready to state our main results in this section.

\begin{prop}\label{building1}
Let $k$ be a nonperfect separably closed field. Let $H$ be a subgroup of $G$. Suppose that $\Delta_k(G)^H$ is contractible but $\Delta(G)^H$ is not contractible. Let $s$ be a simplex in $\Delta_k(G)^H$ of maximal dimension. Let $P$ be the $k$-parabolic subgroup of $G$ corresponding to $s$. Then there exists a unipotent element $u\in R_u(P)(\overline k)$ such that $\Delta_k(G)^{u\cdot H}$ is not contractible.
\end{prop}
\begin{proof}
This follows from Theorem~\ref{Prop1} and Proposition~\ref{serre-G-cr-builidng}.
\end{proof}

\begin{prop}\label{building2}
Let $k$ be a nonperfect separably closed field. Let $H$ be a subgroup of $G$. Suppose that $\Delta_k(G)^H$ is not contractible but $\Delta(G)^H$ is contractible. Let $s$ be a simplex of $\Delta_k(G)^H$ of maximal dimension. Let $P$ be the $k$-parabolic subgroup corresponding to $s$. Then the following hold:
\begin{enumerate}
\item{There exists a $k$-Levi subgroup $L$ of $P$ such that $\Delta_k(L)^H=\emptyset$.}
\item{There exists an element $l\in L(\overline k)$ such that $\Delta_k(G)^{l\cdot H}$ is contractible.}
\end{enumerate}
\end{prop}
\begin{proof}
This follows from Theorem~\ref{Prop2} and Proposition~\ref{serre-G-cr-builidng}.
\end{proof}

\begin{rem}
It would be very interesting to find purely building-theoretic (geometric-topological) proofs of Propositions~\ref{building1} and~\ref{building2}.
\end{rem}
\section{Centralizers and normalizers of $G$-cr subgroups}

In this section, we assume $G$ is connected. Our proof of Theorem~\ref{main} is purely combinatorial and just depends on the structure of the set of roots of $G$ and the corresponding root subgroups of $G$. 

\begin{proof}[Proof of Theorem~\ref{main}]
Suppose that $H$ is $G$-cr over $k$ but $C_G(H)$ is not. Then by~\cite[Prop.~3.3]{Uchiyama-Nonperfect-pre} (which, in turn, depends on the recently proved center conjecture of Tits), there exists a proper $k$-parabolic subgroup of $G$ containing $HC_G(H)$. Let $P$ be a minimal such $k$-parabolic subgroup. Let $P=P_\lambda$ for some $\lambda\in Y_k(G)$. Since $H$ is $G$-cr over $k$, there exists a $k$-Levi subgroup $L$ of $P_\lambda$ containing $H$. Since $k$-Levi subgroups of $P$ are $R_u(P)(k)$-conjugate by~\cite[Lem.~2.5(\rmnum{3})]{Bate-uniform-TransAMS}, we may assume $L=L_\lambda$. Then $\lambda(\overline k^*)<C_G(H)$ since $L_\lambda=C_G(\lambda(\overline k^*))$. 

Suppose that there exists $1\neq u\in R_u(P_\lambda)\cap C_G(H)$. Define $w:=\overline{w_{0,L}}\circ\overline{w_{0,G}}$. Then 
\begin{equation*}
w\cdot(R_u(P_{\lambda}))=\overline{w_{0,L}}\cdot (\overline{w_{0,G}}\cdot R_u(P_{\lambda}))=\overline{w_{0,L}}\cdot (R_u(P_{-\lambda}))=R_u(P_{-\lambda}).
\end{equation*}
Note that for any root $i\in \Psi(L)$ we have 
\begin{equation*}
w\cdot i=\overline{w_{0,L}}\cdot(\overline{w_{0,G}}\cdot i)=\overline{w_{0,L}}\cdot (-i) = i.
\end{equation*}
Then $w$ centralizes $[L,L]$. Since $H$ is semisimple, $H<[L,L]$ so $w$ also centralizes $H$. Thus 
\begin{equation*}
w\cdot (C_G(H))=C_G(w\cdot H)=C_G(H).
\end{equation*}
So we have $1\neq w\cdot u\in R_u(P_{-\lambda})\cap C_G(H)$. This is a contradiction since $C_G(H)<P_\lambda$. So $R_u(P_\lambda)\cap C_G(H)=1$. Thus $HC_G(H)<L$. Since  $C_G(H)$ is not $G$-cr over $k$, it is not $L$-cr over $k$ by Proposition~\ref{G-cr-L-cr}. Then there exists a proper $k$-parabolic subgroup $P_L$ of $L$ containing $HC_G(H)$ by~\cite[Prop~3.3]{Uchiyama-Nonperfect-pre}. Then $P_L\ltimes R_u(P_\lambda)$ is a $k$-parabolic subgroup of $G$ containing $HC_G(H)$ and strictly contained in $P_\lambda$. This contradicts the minimality of $P_\lambda$. Hence $C_G(H)$ is $G$-cr over $k$.
\end{proof}

\begin{rem}\label{remOnAssumptions}
The condition on $\overline{w_{0,L}}$ in the hypothesis of Theorem~\ref{main} is necessary for our proof to work since otherwise the operation of $\overline{w_{0,L}}$ on $R_u(P_{\lambda})$ might not make sense. Consider the following case: $G=A_3$, $\Psi(G)=\{\pm \alpha, \pm \beta, \pm \gamma, \pm(\alpha+\beta), \pm(\beta+\gamma), \pm(\alpha+\beta+\gamma) \}$, $L=L_{\alpha\beta}$. Then $\overline{w_{0,L}}=n_\beta n_\alpha n_\beta \sigma$ where $n_\alpha$, $n_\beta$ are reflections corresponding to $\alpha$, $\beta$ , and $\sigma$ is the non-trivial graph automorphism of $L=A_2$ respectively. We see that $\beta+\gamma\in R_u(P_{\alpha\beta})$ but $\sigma$ cannot be applied to $\beta+\gamma$. 
\end{rem}

Replacing $C_G(H)$ with $N_G(H)$ in the proof of Theorem~\ref{main}, we obtain

\begin{thm}\label{main2}
Let $G$ be connected. Suppose that a semisimple $k$-subgroup $H$ of $G$ is $G$-cr over $k$. Let $P$ be minimal among $k$-parabolic subgroups containing $N_G(H)$, and $L$ be a $k$-Levi subgroup of $P$. If the automorphism $\overline{w_{0,L}}$ of $L$ extends to an automorphism of $G$ and $\overline{w_{0,L}}$ stabilizes $R_u(P)$, then $N_G(H)$ is $G$-cr over $k$. 
\end{thm}

For the converse of Theorem~\ref{main2}, we have 

\begin{prop}
Let $G$ be connected. Let $H$ be a (not necessarily $k$-defined) subgroup of $G$. Suppose that $N_G(H)$ is $G$-cr over $k$. If $N_G(H)$ is defined over $k$, then $H$ is $G$-cr over $k$.
\end{prop} 
\begin{proof}
This is a consequence of \cite[Prop.~3.7]{Uchiyama-Nonperfect-pre}. Note that $N_G(H)$ needs to be $k$-defined to apply \cite[Prop.~3.7]{Uchiyama-Nonperfect-pre}.
\end{proof}
\section{On the number of conjugacy classes}
Let $n$ be a natural number. Let $G$ be a (possibly non-connected) reductive group defined over an algebraically closed field $k$ of characteristic $p\geq 0$. Suppose that $M$ is a (possibly non-connected) reductive subgroup of $G$. Using Richardson's beautiful tangent space argument~\cite{Richardson-Conjugacy-Ann}, Slodowy~\cite{Slodowy-book} showed that if $G$ and $M$ are connected and $p$ is very good for $G$, then $G\cdot (m_1,\cdots, m_n)\cap M^n$ is a finite union of $M$-conjugacy classes where $m_i\in M$ and $G$ acts on $G^n$ via simultaneous conjugation. Moreover Guralnick showed that the same result holds for non-connected $G$ and $M$ for any $p$ if $n=1$~\cite[Thm.~1.2]{Guralnick-conjugacy-procAMS}. 

The following Theorem~\ref{conjugacy} shows two things: even if $p$ is very good for $G^{\circ}$, 1)~Slodowy's result fails if $G$ and $M$ are non-connected, 2)~Guralnick's result fails if $n>1$. Theorem~\ref{conjugacy} works because if $G$ is non-connected of type $A_2$, there exists a nonseparable subgroup of $G$ even if $p$ is very good for $G^{\circ}$. 

\begin{thm}\label{conjugacy}
Let $k$ be an algebraically closed field of characteristic $2$. Let $G=SL_3 \rtimes \langle \sigma \rangle$ be defined over $k$ where $\sigma$ is a non-trivial graph automorphism of $SL_3$.
Then there exists a (non-connected) reductive subgroup $M$ of $G$ and a pair $(m_1, m_2)\in M^2$ such that $G\cdot (m_1, m_2) \cap M^2$ is an infinite union of $M$-conjugacy classes. 
\end{thm}

\begin{proof}
Let $G$ and $\sigma$ be as in the hypothesis. Fix a maximal torus $T$ of $G$ and a Borel subgroup $B$ of $G^{\circ}$ containing $T$. Let $\Psi(G)^{+}=\{\alpha, \beta, \alpha+\beta\}$ be the set of positive roots of $G$ with respect to $T$ and $B$. 
Then $\sigma\cdot \alpha = \beta$ and $\sigma\cdot \beta$. In the following we use the commutation relations~\cite[Sec.~33.3]{Humphreys-book1} repeatedly. 

Let $\lambda:=(\alpha+\beta)^{\vee}\in Y(G)$. Then $P_\lambda=\langle T, U_{\alpha}, U_\beta, U_{\alpha+\beta}, \sigma \rangle$, and $R_u(P_\lambda)=\langle U_{\alpha}, U_\beta, U_{\alpha+\beta}\rangle$. Let $\epsilon_{\alpha}(x)\epsilon_{\beta}(y)\epsilon_{\alpha+\beta}(z)$ be an arbitrary element in $R_u(P_\lambda)$. For $x, y, z \in k$, we have
\begin{alignat*}{2}
\sigma\cdot (\epsilon_{\alpha}(x)\epsilon_{\beta}(y)\epsilon_{\alpha+\beta}(z)) &= 
\epsilon_{\beta}(x)\epsilon_{\alpha}(y)\epsilon_{\alpha+\beta}(z) \\
&=\epsilon_{\alpha}(y)\epsilon_{\beta}(x)\epsilon_{\alpha+\beta}(xy+z).
\end{alignat*}
Thus $\epsilon_{\alpha}(x)\epsilon_{\beta}(y)\epsilon_{\alpha+\beta}(z)$ is not in $C_G(\sigma)$ unless $x=y=0$. In particular, the curve $\{ \epsilon_{\alpha}(x)\epsilon_{\beta}(x)\mid x\in k\}$ is not in $C_G(H)$. 
On the other hand, we have 
\begin{equation*}
\sigma\cdot (e_{\alpha}+e_{\beta}) = e_{\alpha}+e_{\beta}.
\end{equation*}
The same argument as in Remark~\ref{D4nonsep} shows that $\langle \sigma \rangle$ is non-separable in $G$. 

Now let $v(x):=\epsilon_{\alpha}(x)\epsilon_{\beta}(x)$ for $x\in k$. Clearly $v(x)\in R_u(P_\lambda)$. Let $M:=\langle \sigma, G_{\alpha+\beta}\rangle$. Then $M$ is (non-connected) reductive. We have $Z(R_u(P_\lambda))=U_{\alpha+\beta}$. Let $(m_1, m_2):=(\sigma, \epsilon_{\alpha+\beta}(1))$. Define 
\begin{equation*}
(m_1, m_2)_x:=v(x)\cdot (m_1, m_2) = (\sigma\epsilon_{\alpha+\beta}(x^2), \epsilon_{\alpha+\beta}(1)).
\end{equation*}
We see that $(m_1, m_2)_x\in M^2$. Note that $\sigma$ centralizes $(m_1, m_2)_x$. Since $M^{\circ}=SL_2$, a direct computation shows that if $a\neq b$, $(m_1, m_2)_a$ is not $M$-conjugate to $(m_1, m_2)_b$.
\end{proof}

\begin{rem}
The same example was used to negatively answer a question of K\"ulshammer on representations of finite groups~\cite[Thm.~1.14]{Uchiyama-Classification-pre}. We believe that the conjugacy class problem in this section is closely related to K\"ulshammer's question. See~\cite{Bate-QuestionOfKulshammer} and~\cite{Martin-Lond-Kulshammer-Arx} for more on this. 
\end{rem}

\begin{rem}
In the proof of Theorem~\ref{conjugacy}, $p=2$ is crucial. Even if $p=3$ is bad for our $G$, if $p=3$ our argument breaks down. This is because when $p=3$ the subgroup $\langle \sigma \rangle$ is linearly reductive (since the order of $\sigma$ is not divisible by $p$) and linearly reductive subgroups are separable in $G$.  
\end{rem}

\section*{Acknowledgements}
This research was supported by Marsden Grant UOA1021 and a postdoctoral fellowship at the National Center for Theoretical Sciences at the National Taiwan University. The author would like to thank Michel Brion, Brian Conrad, Philippe Gille, Benjamin Martin, and Jean-Pierre Serre for helpful discussions and detailed comments. 
\bibliography{mybib}

\end{document}